\documentclass[12pt]{amsart}
\usepackage{amsfonts,amssymb,amscd,amsmath,enumerate,verbatim}
\usepackage[latin1]{inputenc}
\usepackage{amscd}
\usepackage{latexsym}
\usepackage{color}
\usepackage{mathptmx}
\usepackage{multicol}
\input xy
\xyoption{all}
\def\NZQ{\mathbb}               

\def\ZZ{{\NZQ Z}}
\def\RR{{\NZQ R}}

\def\frk{\mathfrak}               

\def\Phi{{\frk N}}
\def\eb{{\bold e}}

\def\xb{{\bold x}}

\def\opn#1#2{\def#1{\operatorname{#2}}} 
\opn\chara{char} 
\opn\length{\ell} 
\opn\pd{pd} 
\opn\rk{rk}
\opn\projdim{proj\,dim} 
\opn\injdim{inj\,dim} 
\opn\rank{rank}
\opn\depth{depth} 
\opn\grade{grade} 
\opn\height{height}
\opn\embdim{emb\,dim} 
\opn\codim{codim}

\opn\Tr{Tr} 
\opn\bigrank{big\,rank}
\opn\superheight{superheight}
\opn\lcm{lcm}
\opn\trdeg{tr\,deg}
\opn\reg{reg} 
\opn\lreg{lreg} 
\opn\ini{in} 
\opn\lpd{lpd}
\opn\size{size}
\opn\mult{mult}
\opn\dist{dist}
\opn\cone{cone}
\opn\lex{lex}
\opn\rev{rev}
\opn\div{div} \opn\Div{Div} \opn\cl{cl} \opn\Cl{Cl}
\opn\Spec{Spec} \opn\Supp{Supp} \opn\supp{supp} \opn\Sing{Sing}
\opn\Ass{Ass} \opn\Min{Min}
\opn\Ann{Ann} \opn\Rad{Rad} \opn\Soc{Soc}
\opn\Syz{Syz} \opn\Im{Im} \opn\Ker{Ker} \opn\Coker{Coker}
\opn\Am{Am} \opn\Hom{Hom} \opn\Tor{Tor} \opn\Ext{Ext}
\opn\End{End} \opn\Aut{Aut} \opn\id{id} \opn\ini{in}

\opn\nat{nat}
\opn\pff{pf}
\opn\Pf{Pf} \opn\GL{GL} \opn\SL{SL} \opn\mod{mod} \opn\ord{ord}
\opn\Gin{Gin}
\opn\Hilb{Hilb}\opn\adeg{adeg}\opn\std{std}\opn\ip{infpt}
\opn\Pol{Pol}
\opn\sat{sat}
\opn\Var{Var}
\opn\Gen{Gen}

\opn\aff{aff} \opn\con{conv} \opn\relint{relint} \opn\st{st}
\opn\lk{lk} \opn\cn{cn} \opn\core{core} \opn\vol{vol}
\opn\link{link} \opn\star{star}
\opn\gr{gr}

\def\Pc{{\mathcal P}}
\def\Qc{{\mathcal Q}}

\def\vol{{\textnormal{vol}}}
\def\conv{{\textnormal{conv}}}

\def\ord{{\textnormal{ord}}}
\def\pot#1#2{#1[\kern-0.28ex[#2]\kern-0.28ex]}

\opn\dirlim{\underrightarrow{\lim}}
\opn\inivlim{\underleftarrow{\lim}}
\let\to=\rightarrow

\def\Implies{\ifmmode\Longrightarrow \else
	\unskip${}\Longrightarrow{}$\ignorespaces\fi}
\def\implies{\ifmmode\Rightarrow \else
	\unskip${}\Rightarrow{}$\ignorespaces\fi}
\def\iff{\ifmmode\Longleftrightarrow \else
	\unskip${}\Longleftrightarrow{}$\ignorespaces\fi}

\let\:=\colon
\newtheorem{Theorem}{Theorem}[section]
\newtheorem{Lemma}[Theorem]{Lemma}
\newtheorem{Corollary}[Theorem]{Corollary}

\newtheorem{Example}[Theorem]{Example}

\numberwithin{equation}{section}
\newtheorem*{acknowledgement}{Acknowledgment}
\let\epsilon\varepsilon
\let\phi=\varphi
\let\kappa=\varkappa
\def\qed{\ifhmode\textqed\fi
	\ifmmode\ifinner\quad\qedsymbol\else\dispqed\fi\fi}
\def\textqed{\unskip\nobreak\penalty50
	\hskip2em\hbox{}\nobreak\hfil\qedsymbol
	\parfillskip=0pt \finalhyphendemerits=0}
\def\dispqed{\rlap{\qquad\qedsymbol}}

\opn\dis{dis}
\opn\height{height}
\opn\dist{dist}
\def\pnt{{\raise0.5mm\hbox{\large\bf.}}}

\opn\Lex{Lex}
\opn\conv{conv}

\begin{document}
\title{Ehrhart polynomials of lattice polytopes with normalized volumes $5$}
\author[A.~Tsuchiya]{Akiyoshi Tsuchiya}
\address[Akiyoshi Tsuchiya]{Department of Pure and Applied Mathematics,
	Graduate School of Information Science and Technology,
	Osaka University,
	Suita, Osaka 565-0871, Japan}
\email{a-tsuchiya@ist.osaka-u.ac.jp}
\subjclass[2010]{52B12, 52B20}
\keywords{lattice polytope, $\delta$-polynomial, $\delta$-vector, Ehrhart polynomial, spanning polytope}
\begin{abstract}
	A complete classification of the $\delta$-vectors of
	lattice polytopes whose normalized volumes are at most $4$ is known.
	In the present paper, we will classify all the $\delta$-vectors of lattice polytopes with normalized volumes $5$.
\end{abstract} 

\maketitle 
\section*{Introduction}
One of the final, however, unreachable goal of the study on lattice polytopes is to classify lattice polytopes up to unimodular equivalence.
In lower dimension, lattice polytopes with a small volume    are classified (\cite{Bal})
and lattice polytopes with a small number of lattice points are classified (\cite{santos1, santos2}).
On the other hand, for arbitrary dimension, all lattice polytopes whose normalized volumes are at most $4$ are completely classified (\cite{HT}). 
In order to do this task, a complete classification of the $\delta$-vectors of lattice polytopes whose normalized volumes are at most $4$ is used.
This implies that finding a combinatorial characterization of the $\delta$-vectors of lattice polytopes is useful for classifying lattice polytopes.
In the present paper, as a next step, we will classify all the $\delta$-vectors of lattice polytopes whose normalized volumes are $5$.
\subsection{Background on $\delta$-vectors}
First, recall from \cite[Part II]{HibiRedBook} what   $\delta$-vectors are.
We say that a convex polytope is a \textit{lattice polytope}
if its vertices are all elements in $\ZZ^d$.
Let $\mathcal{P},\mathcal{Q} \subset \RR^d$  be  lattice polytopes of dimension $d$.
We say that $\mathcal{P}$ and $\mathcal{Q}$ are  {\em unimodularly equivalent} if there exists an unimodular 
transformation that maps on one polytope to the other, that is, an affine map $f : \RR^d \to \RR^d$ with $f(\ZZ^d)=\ZZ^d$ and $f(\Pc)=\Qc$.
In this case, we write $\Pc \cong \Qc$.
Given a positive integer $n$, we define
$$L_{\Pc}(n)=|n \Pc \cap \ZZ^d|,$$
where $n\Pc=\{n \xb : \xb \in \Pc\}$ and $|X|$ is the cardinality of a finite set $X$. 
The study on $L_{\Pc}(n)$ originated in Ehrhart \cite{Ehrhart} who proved that $L_{\Pc}(n)$ is a polynomial in $n$ of degree $d$ with the constant term $1$.
Furthermore, the leading coefficient, that is, the coefficient of $n^d$ of $L_{\Pc}(n)$ coincides with the usual volume of $\Pc$.
We say that $L_{\Pc}(n)$ is the \textit{Ehrhart polynomial} of $\Pc$.
Clearly, if $\Pc \cong \Qc$, then one has $L_{\Pc}(n)=L_{\Qc}(n)$.

We define $\delta(\Pc,t)$ by the formula
$$ \delta(\Pc,t)=(1-t)^{d+1}\left[1+\sum_{n=1}^{\infty}L_{\Pc}(n)t^n\right].$$
Then it follows that $\delta(\Pc,t)$ is a polynomial in $t$ of degree at most $d$.
Set $\delta(\Pc,t)=\delta_0+\delta_1t+\cdots+\delta_dt^d$.
We say that $\delta(\Pc,t)$ is the \textit{$\delta$-polynomial} and the sequence $(\delta_0,\ldots,\delta_d)$ is the \textit{$\delta$-vector} of $\Pc$.
The following properties of $\delta(\Pc,t)$ are known:
\begin{itemize}
	\item $\delta_0=1$, $\delta_1=|\Pc \cap \ZZ^d|-(d+1)$ and $\delta_d=|(\Pc \setminus \partial \Pc) \cap \ZZ^d|$, where $\partial \Pc$ is the boundary of $\Pc$. Hence one has $\delta_1 \geq \delta_d$;
	\item $\delta_i \geq 0$ for each $i$;
	\item When $\delta_d \neq 0$, one has $\delta_i \geq \delta_1$ for $1 \leq i \leq d-1$;
	\item $\delta(\Pc,1)=\sum_{i=0}^{d}\delta_i$ coincides with the \textit{normalized volume} of $\Pc$.
\end{itemize} 

There are two well-known inequalities on $\delta$-vectors.
Let $s$ be the degree of the $\delta$-polynomial, i.e., $s=\textnormal{max}\{i : \delta_i \neq 0\}$.
In \cite{Stanley_ineq}, Stanley proved that
\begin{equation}
\label{eq1}
\delta_0+\delta_1+\cdots+\delta_i \leq \delta_s+\delta_{s-1}+\cdots+\delta_{s-i}, \ \ 0 \leq i \leq \lfloor s/2 \rfloor,
\end{equation}
while in \cite{Hibi_ineq}, Hibi proved that
\begin{equation}
\label{eq2}
\delta_{d-1}+\delta_{d-2}+\cdots+\delta_{d-i} \leq \delta_2+\delta_{3}+\cdots+\delta_{i+1}, \ \ 1 \leq i \leq \lfloor (d-1)/2 \rfloor.
\end{equation}
Recently, there are more general results of inequalities on $\delta$-vectors by Stapledon in \cite{Stap1,Stap2}.

\subsection{Characterization of $\delta$-vectors with small volumes}
One of the most fundamental problems of enumerative combinatorics is to find a combinatorial characterization of all vectors that can be realized as the $\delta$-vector of some lattice polytope. 
For example, restrictions like $\delta_0=1$,  $\delta_i \geq 0$, and the inequalities (\ref{eq1}) and (\ref{eq2}) are necessary conditions for a vector to be the $\delta$-vector of some lattice polytope.
On the other hand, in \cite{HHN}, the possible $\delta$-vectors with $\delta_0+\cdots+\delta_d \leq 3$ are completely classified by the inequalities (\ref{eq1}) and (\ref{eq2}).

\begin{Theorem}[{\cite[Theorm 0.1]{HHN}}]
	\label{HHN}
	Let $d \geq 3$.
	Given a sequence $(\delta_0,\ldots,\delta_d)$ of nonnegative integers, 
	where $\delta_0=1$ and $\delta_1 \geq \delta_d$, 
	which satisfies $\sum_{i=0}^{d}\delta_{i} \leq 3$, 
	there exists a lattice polytope $\Pc \subset \RR^d$ of dimension $d$ whose $\delta$-vector coincides with $(\delta_0,\ldots,\delta_d)$
	if and only if $(\delta_0,\ldots,\delta_d)$ satisfies all inequalities {\rm(\ref{eq1})} and {\rm (\ref{eq2})}.
 \end{Theorem}
However, Theorem \ref{HHN} is not true for $\delta_0+\cdots+\delta_d =4$ (see \cite[Example 1.2]{HHN}).
On the other hand,
in \cite[Theorem 5.1]{HHL},  a complete classification of the possible $\delta$-vectors with  $\delta_0+\cdots+\delta_d =4$ is given.
\begin{Theorem}[{\cite[Theorem 5.1]{HHL}}]
	Let $1+t^{i_1}+t^{i_2}+t^{i_3}$
	be a polynomial with $1 \leq i_1 \leq i_2 \leq i_3 \leq d$.
	Then there exists a lattice polytope $\Pc \subset \RR^d$ of dimension $d$ whose $\delta$-polynomial equals $1+t^{i_1}+t^{i_2}+t^{i_3}$ if and only if $(i_1,i_2,i_3)$ satisfies 
	\[
	i_3 \leq i_1+i_2, i_1+i_3 \leq d+1 \ {\rm and} \ i_2 \leq \lfloor (d+1)/2\rfloor,
	\]
	and the additional conditions 
	\[
	2i_2 \leq i_1+i_3 \ {\rm or} \ i_2+i_3 \leq d+1.
	\]
	Moreover, all these polytopes can be chosen to be simplices.
\end{Theorem}
We remark that there exists a sequence  $(\delta_0,\ldots,\delta_d)$
of nonnegative integers such that $(\delta_0,\ldots,\delta_d)$ is not the $\delta$-vector of any lattice simplex but it is the $\delta$-vector of some lattice non-simplex
 (\cite[Remark 5.3]{HHL}).

\subsection{Main result: characterization of $\delta$-vectors with $\sum_{i=0}^{d} \delta_i =5$ }
In \cite{Higprime},
Higashitani classified all the possible $\delta$-vectors of lattice 
simplices whose normalized volumes are $5$.
\begin{Theorem}[{\cite[Theorem 1.2]{Higprime}}]
	\label{Hig5}
	Let $1+t^{i_1}+t^{i_2}+t^{i_3}+t^{i_4}$ be a polynomial with some positive integers $ i_1 \leq \cdots \leq i_4 \leq d$.
	Then there exists a lattice simplex of dimension $d$ whose $\delta$-polynomial equals $1+t^{i_1}+t^{i_2}+t^{i_3}+t^{i_4}$ if and only if the following conditions are satisfied:
	\begin{itemize}
		\item $i_1+i_4=i_2+i_3 \leq d+1$;
		\item $i_k + i_{\ell} \geq i_{k+\ell}$ for $1 \leq k \leq \ell \leq 4$ with $k+\ell \leq 4$.
	\end{itemize}
	\end{Theorem}

In the present paper, we will classify all the possible $\delta$-vectors of lattice 
polytopes whose normalized volumes are $5$.
In fact, we will show the following theorem.
\begin{Theorem}
	\label{main}
Let $1+t^{i_1}+t^{i_2}+t^{i_3}+t^{i_4}$ be a polynomial with some positive integers $ i_1 \leq \cdots \leq i_4 \leq d$.
Then there exists a lattice polytope of dimension $d$ whose $\delta$-polynomial equals $1+t^{i_1}+t^{i_2}+t^{i_3}+t^{i_4}$ if and only if $(i_1,i_2,i_3,i_4)$ satisfies the condition of Theorem {\rm \ref{Hig5}} or one of the following conditions:
\begin{enumerate}
	\item $(i_1,i_2,i_3,i_4)=(1,1,1,2)$ and $d \geq 2$;
	\item $(i_1,i_2,i_3,i_4)=(1,2,2,2)$ and $d \geq 3$;
	\item $(i_1,i_2,i_3,i_4)=(1,2,3,3)$ and $d \geq 5$.
\end{enumerate}
In particular, we cannot obtain the $\delta$-polynomials of $(1)$, $(2)$ and $(3)$ by lattice simplices.
\end{Theorem}

\subsection{Structure of this paper}
The present paper is organized as follows:
First, in Section $1$, we will discuss some properties of lattice polytopes whose normalized volumes are prime integers.
In particular, we will show that every lattice polytope of  which is not an empty simplex and whose normalized volume equals a prime integer is always a spanning polytope (Theorem \ref{span5}).
This is a key result in the present paper.
Finally, in Section $2$, by using this result we will prove Theorem \ref{main}.

\begin{acknowledgement}{\rm
		The author would like to thank anonymous referees for reading the manuscript carefully.
		The author is partially supported by Grant-in-Aid for JSPS Fellows 16J01549.
	}
\end{acknowledgement}

\section{Lattice polytopes with prime volumes}
In this section, we will discuss some properties of lattice polytopes whose normalized volumes are prime integers.

Let $\Pc \cap \ZZ^d$ be a lattice polytope of dimension $d$ and 
$\langle \Pc \cap \ZZ^d \rangle_{\ZZ}$ the affine sublattice generated by $\Pc \cap \ZZ^d$.
We call the \textit{index} of $\Pc$ the index of 
$\langle \Pc \cap \ZZ^d \rangle_{\ZZ}$ as a sublattice of $\ZZ^d$.
We say that $\Pc$  is 
 \textit{spanning} if its index equals $1$.
This is equivalent to that any lattice point in $\ZZ^{d+1}$ is a linear integer combination of the lattice points in $\Pc \times \{1\}$.
A lattice simplex is called \textit{empty} if it has no lattice point expect for its vertices.
Now, we prove the following theorem.
\begin{Theorem}
	\label{span5}
	Let $p$ be a prime integer and
	$\Pc \subset \RR^d$ be a lattice polytope of dimension $d$  whose normalized volume equals $p$.
	Suppose that $\Pc$ is not an empty simplex. Then $\Pc$ is spanning.
\end{Theorem}
\begin{proof}
	Since $\Pc$ is not an empty simplex,
	there exists a lattice triangulation $\{\Delta_1,\ldots,\Delta_k \}$  of $\Pc$ with  some positive integer $k \geq 2$.
	Since the index of
	$P$ must divide the normalized volume of
	every $\Delta_i$, and since the sum of those normalized volumes is the prime integer
	$p$, the index must be one.
	Hence $\Pc$ is spanning.
\end{proof}

Next, we consider an application of this result to classifying lattice polytopes whose normalized volumes are prime integers.
Thanks to Theorem \ref{span5},
every full-dimensional lattice polytope whose normalized volumes equals $5$ is either an empty simplex or a spanning polytope.
See e.g., \cite{HZ} for how to classify empty simplices.
Now, we focus on spanning polytopes.
For a lattice polytope $\mathcal{P} \subset \RR^d$,
the {\em lattice pyramid} over $\mathcal{P}$ is defined by $\text{conv}(\mathcal{P}\times \left\{ 0 \right\} ,(0,\ldots,0,1))$ $\subset \RR^{d+1}$. We denote this by $\text{Pyr}(\mathcal{P})$. 
Let us recall the following result.
\begin{Lemma}[{\cite[Corollary 2.4]{HKN}}]
	\label{spanclass}
	There are only finitely many spanning lattice polytopes of given normalized volume (and arbitrary dimension) up to unimodular equivalence and lattice pyramid constructions.
\end{Lemma}

By combining Theorem \ref{span} and Lemma \ref{spanclass}, we can obtain the following corollary.
\begin{Corollary}
	Let $p$ be a prime integer and $\Pc$ a  lattice polytope of dimension $d$ whose normalized volume equals $p$.
	Suppose that $\Pc$ is not an empty simplex.
	Then there are only finitely many possibilities for $\Pc$ up to unimodular equivalence and lattice pyramid constructions.
\end{Corollary}

\section{Proof of Theorem \ref{main}}

In this section we will prove Theorem \ref{main}.
First, recall the following lemmas.
\begin{Lemma}[\cite{Bat}]
	\label{pyr}
	Let $\Pc \subset \RR^d$ be a lattice polytope of dimension $d$.
	Then one has
	$$\delta(\textnormal{Pyr}(\Pc),t)=\delta(\Pc,t).$$
\end{Lemma}

\begin{Lemma}[{\cite[Theorem 1.3]{HKN}}]
	\label{span}
	Let
	$\Pc \subset \RR^d$ be a lattice polytope of dimension $d$ whose $\delta$-polynomial equals $\delta_0+\delta_1t+\cdots +\delta_st^s$, where $\delta_s \neq 0$.
	If $\Pc$ is spanning,
	then one has $\delta_i \geq 1$ for any $0 \leq i \leq s$.
\end{Lemma}

By combining Theorem \ref{span5} and Lemma \ref{span}, we can obtain the following corollary.
\begin{Corollary}
	\label{deltaprime}
Let $p$ be a prime integer and $\Pc \subset \RR^d$ a  lattice polytope of dimension $d$ whose normalized volume equals $p$ and whose $\delta$-polynomial equals $\delta_0+\delta_1t+\cdots+\delta_st^s$, where $\delta_s \neq 0$.	
Suppose that $\Pc$ is not an empty simplex.
Then one has $\delta_i \geq 1$ for any $0 \leq i \leq s$.
\end{Corollary}

Next, we give indispensable examples for our proof of Theorem \ref{main}.

\begin{Example}\label{ex}
	{\em
(a) Let $\Pc_1 \subset \RR^2$ be the lattice polytope which is the convex hull of the following lattice points:
$${\bf 0}, \eb_1, \eb_2, 2\eb_1+3\eb_2 \in \RR^2.$$
Then one has $\delta(\Pc_1,t)=1+3t+t^2$.

(b) Let $\Pc_2 \subset \RR^3$ be the lattice polytope which is the convex hull of the following lattice points:
$${\bf 0},\eb_1,\eb_2,\eb_3,  \eb_1+\eb_2+3\eb_3 \in \RR^3.$$
Then one has $\delta(\Pc_2,t)=1+t+3t^2$.

(c) Let $\Pc_3 \subset \RR^5$ be the lattice polytope which is the convex hull of the following lattice points:
$${\bf 0},\eb_1,\eb_2,\eb_3,\eb_4,\eb_5,-\eb_1+\eb_2+\eb_3+\eb_4+2\eb_5 \in \RR^5.$$
Then one has $\delta(\Pc_3,t)=1+t+t^2+2t^3$.
}
\end{Example} 

Finally, we prove Theorem \ref{main}.
\begin{proof}[Proof of Theorem \ref{main}]
First, we can prove the "If"  part of Theorem \ref{main} from Theorem \ref{Hig5}, Lemma \ref{pyr} and Example \ref{ex}.
Hence we should prove the "Only if " part of Theorem \ref{main}.
Let $\Pc \subset \RR^d$ be a lattice non-simplex of dimension $d$ whose normalized volume equals $5$
and $\delta(\Pc,t)=\delta_0+\delta_1t+\cdots+\delta_dt^d$ the $\delta$-polynomial of $\Pc$.
By Corollary \ref{deltaprime}  and the inequalities (\ref{eq1}) and (\ref{eq2}), and the fact $\delta_1 \geq \delta_d$, one of the followings is satisfied:
\begin{enumerate}
	\item $\delta(\Pc,t)=1+4t$ and $d \geq 1$;
	\item $\delta(\Pc,t)=1+3t+t^2$ and $d \geq 2$;
	\item $\delta(\Pc,t)=1+2t+2t^2$ and $d \geq 2$;
	\item $\delta(\Pc,t)=1+t+3t^2$ and $d \geq 3$;
	\item $\delta(\Pc,t)=1+t+2t^2+t^3$ and $d \geq 3$;
	\item $\delta(\Pc,t)=1+t+t^2+2t^3$ and $d \geq 5$;
	\item $\delta(\Pc,t)=1+t+t^2+t^3+t^4$ and $d \geq 4$.
\end{enumerate}
Then we know that the conditions $(1),(3),(5)$ and $(7)$ satisfy the condition of Theorem \ref{Hig5}.
This completes the proof.
\end{proof}

	

\end{document}